\documentclass[11pt]{amsart}

\usepackage{amsthm,amssymb,amstext,amscd,amsfonts,amsbsy,amsxtra,latexsym,amsmath}
\usepackage{fullpage}
\usepackage[english]{babel}
\usepackage[latin1]{inputenc}
\usepackage{verbatim}
\usepackage{mathrsfs}
\usepackage[all]{xy}
\newcommand{\field}[1]{\mathbb{#1}}

\newcommand{\Z}{\field{Z}}

\newcommand{\Q}{\field{Q}}

\newcommand{\sgn}{\operatorname{sgn}}

\newcommand{\SL}{\operatorname{SL}}

\DeclareMathOperator*{\Res}{Res}
\newcommand{\im}{\text{Im}}

\newcommand{\e}{e}

\renewcommand{\epsilon}{\varepsilon}
\renewcommand{\theta}{\vartheta}
\renewcommand{\rho}{\varrho}
\renewcommand{\phi}{\varphi}
\numberwithin{equation}{section}
\newtheorem{theorem}{\textbf{Theorem}}
\numberwithin{theorem}{section}
\newtheorem{lemma}[theorem]{\textbf{Lemma}}
\newtheorem{proposition}[theorem]{\textbf{Proposition}}

\newtheorem{definition}[theorem]{\textbf{Definition}}
\theoremstyle{remark}
\newtheorem{remark}{Remark}

\newcommand{\rl}[1]{\left( #1 \right)}
\newcommand{\rlf}[2]{\left(\frac{#1}{#2}\right)}
\newcommand{\ql}[1]{\left[ #1 \right]}

\renewenvironment{proof}[1][Proof]{\begin{trivlist}
\item[\hskip \labelsep {\bfseries #1:}]}{\qed\end{trivlist}}

\usepackage[babel]{csquotes}

\begin{document}
\title{On the Fourier coefficients of meromorphic Jacobi forms}
\author{Ren\'e Olivetto}
\address{Mathematical Institute\\University of Cologne\\ Weyertal 86-90 \\ 50931 Cologne \\Germany}
\email{rolivett@math.uni-koeln.de}
\thanks{The research of the author is supported by Graduiertenkolleg ``Global Structures in Geometry and Analysis". This paper is part of the author's Ph.D. thesis, written under the supervision of K. Bringmann.}
\date{\today}
\keywords{meromorphic Jacobi forms, almost harmonic Maass forms, canonical Fourier coefficients}
\subjclass[2010] {}
\let\thefootnote\relax\footnote{}
\begin{abstract}
In this paper, we describe the automorphic properties of the Fourier coefficients of meromorphic Jacobi forms. Extending results of Dabholkar, Murthy, and Zagier \cite{DMZ}, and Bringmann and Folsom \cite{BF2}, we prove that the canonical Fourier coefficients of a meromorphic Jacobi form $\phi(z;\tau)$ are the holomorphic parts of some (vector-valued) almost harmonic Maass forms. We also give a precise description of their completions, which turn out to be uniquely determined by the Laurent coefficients of $\phi$ at each pole, as well as some well known real analytic functions, that appear for instance in the completion of Appell-Lerch sums \cite{Zwegers:PhD}.
\end{abstract}

\maketitle
\section{Introduction and statement of results}
The theory of Jacobi forms was first extensively studied by Eichler and Zagier \cite{MR781735}. A Jacobi form is an automorphic form in two or more variables, one of which is defined on $\mathbb{H}$, usually denoted by $\tau$, and is called the modular variable. The other variables are called elliptic, which are defined in $\mathbb{C}$ and in the case that there is precisely one, we denote it by $z$. As well as automorphic and elliptic  properties (for details see Definition \ref{defi:5}), Jacobi forms must satisfy a growth condition and they are required to be holomorphic in both of the variables.

As described in detail in \cite{MR781735} the space of Jacobi forms of given index and weight is isomorphic to a certain space of vector-valued modular forms of half-integral weight. More precisely, given a holomorphic Jacobi form of weight $k\in \Z$ and index $m \in \mathbb{N}$, it is possible to write it in the form
\begin{equation}\label{eq:1}
\varphi\left(z;\tau \right)=\sum_{\ell \pmod{2m}}h_{\ell}(\tau)\vartheta_{m,\ell}\left(z;\tau \right),
\end{equation} 
where the functions $\vartheta_{m,\ell}$, independent of $\phi$, are the usual level $m$ Jacobi theta functions \eqref{eq:8}, while the functions $h_{\ell}$ are essentially the Fourier coefficients of $\phi$ as a function of $z$ \eqref{eq:9}. Because of the modular properties of $\phi$ and $\theta_{m,\ell}$, it follows that $\left(h_{\ell}\right)_{\ell \pmod{2m}}$ is  a vector-valued modular form of weight $k-\frac{1}{2}$ for $\SL_2(\Z)$. Decomposition \eqref{eq:1} is called the \textit{theta decomposition} of $\phi$.

During the last decade the study of other fascinating mathematical objects, including mock theta functions \cite{Zwegers:PhD} and Kac-Wakimoto characters \cite{BF2}, put the attention on meromorphic Jacobi forms. Such a function satisfies all the transformation laws of a classical Jacobi form, but it is allowed to be meromorphic as a function of $z$, and weakly holomorphic in $\tau$, i.e., with no poles in the upper half plane.
A theta decomposition as in \eqref{eq:1} clearly does not make sense in the meromorphic case, since the right hand side is holomorphic in $z$. Since $h_{\ell}$ should be morally thought of as the Fourier coefficients of $\phi$, this does not make sense in general if a function is meromorphic. On the other hand, it is possible to consider a local Fourier expansion of $\phi$ by defining the so-called ``canonical Fourier coefficients", and to define the ``finite part" $\phi^F$ of $\phi$ in terms them. The difference between $\phi$ and $\phi^F$ essentially consists of the contribution of the residues of the poles of $\phi$, since it represents the difference between the Fourier expansions defined in two different regions. For this reason it is called the ``polar part".

The idea of splitting a meromorphic Jacobi form into two parts was originally due to Zwegers \cite{Zwegers:PhD}. Motivated by the fact that most of the fifth-order mock theta functions can be given as Fourier coefficients of certain quotients of ordinary Jacobi theta series, Zwegers investigated  the modularity of the Fourier coefficients of meromorphic Jacobi forms. In order to do so, he gave a decomposition of a meromorphic Jacobi form into two non-holomorphic pieces, one of which was given as a theta decomposition as in \eqref{eq:1}, whose Fourier coefficients turned out to be non-holomorphic modular forms.

Another important field where meromorphic Jacobi forms play a fundamental role is the quantum theory of black holes. In their recent paper, Dabholkar, Murthy, and Zagier \cite{DMZ} explained the relation between the wall-crossing phenomenon and the mock modularity of the generating functions of certain quantum black holes. Of particular interest to the purposes of this paper, they described the structure of meromorphic Jacobi forms with poles of order at most two. Their work improved upon Zwegers's work by separating $\phi$ into two natural parts; one holomorphic and well understood and the other constructed directly from the poles of $\phi$. To be more precise, they constructed a canonical decomposition
\begin{equation}\label{eq:2}
\phi=\phi^F+\phi^P
\end{equation}
where $\phi^F$ turned out to be a mock Jacobi form (see Section 3 for the definition), and $\phi^P$ was a finite linear combination of Appell-Lerch sums with modular forms as coefficients. Decomposition \eqref{eq:2} was defined as allude to before. They then completed explicitly both of the pieces obtaining a new decomposition
\begin{equation} \label{eq:3}
\phi=\widehat{\phi}^F+\widehat{\phi}^P,
\end{equation}
where both of the summands satisfied the same modular property of the original Jacobi form $\phi$. Equivalently, they proved that the canonical Fourier coefficients $h_{\ell}$ (defined in \eqref{eq:23}) were mock modular forms.


A more general construction of a completion for $\phi^F$ led to a new type of modular object, as Bringmann and Folsom discovered in \cite{BF2}. Considering the irreducible $s\ell(m|n)^\wedge$ modules of highest weight $\Lambda(\ell)$ for the affine Lie superalgebra $s\ell(m|n)^{\wedge}$, Bringmann and Folsom investigated the modularity of its specialized characters $tr_{L_{m,n}(\Lambda(\ell))}$ (given by Kac and Wakimoto in \cite{MR1810948}), extending a previous result of Bringmann and Ono \cite{MR2534107} for the special case $n=1$. This led them to study the structure of the generating function
\[
tr_{L_{m,n}}:=\sum_{\ell \in \Z} tr_{L_{m,n}(\Lambda(\ell))}e^{2\pi i\ell z},
\]
for $-\im(\tau)<\im(z)<0$, which turned out to be essentially the meromorphic Jacobi form  
\[
\phi\left(z;\tau\right):=\frac{\theta\left( z+\frac{1}{2};\tau\right)^m }{\theta\left( z;\tau\right)^n} ,
\]
where $m,n \in \mathbb{Z}$ and $\theta\left(z;\tau\right):=\sum_{\nu \in \frac{1}{2}+\mathbb{Z}}e^{\pi i \nu ^2 \tau+2\pi i \nu\left( z+\frac{1}{2} \right) }$ is the classical Jacobi theta function. 
Notice that in this particular case the function $\phi$ has the advantage to have poles in $\Z \tau +\Z$. On the other hand, the poles have arbitrary order $n$, which makes the situation more complicated respect to the case $n\leq 2$. Bringmann and Folsom's investigation of the functions $h_{\ell}$ led to the definition of a new type of modular object. In essence, they showed that the canonical Fourier coefficients can be completed to obtain ``almost harmonic Maass forms" (see Definition \ref{defi:6}). 

Using the same approach as in \cite{BF2} we prove that any meromorphic Jacobi form with poles in $\Q\tau+\Q$ allows a decomposition as in \eqref{eq:3} and that the canonical Fourier coefficients $h_{\ell}$ can be canonically completed to almost harmonic Maass forms. This naturally generalizes the results in \cite{BF2} and \cite{DMZ}.

\begin{theorem} \label{thm:1}
The canonical Fourier coefficients of a meromorphic Jacobi form with poles in $\Q \tau +\Q$ are the holomorphic parts of vector--valued almost harmonic Maass forms.
\end{theorem}
\begin{remark}
The proof of Theorem \ref{thm:1} will be done assuming integral weight and positive integral index. Furthermore, it is assumed that the modularity is satisfied for the full modular group $\SL_2(\Z)$. We point out that the result extends naturally to half--integral weight, half--integral index, and for congruence subgroups. An example of the more general case in described in the author's Ph.D. thesis \cite{Oliv:PhD}. 
\end{remark}

The remainder of the paper is organized as follows. In Section 2, we describe the main automorphic objects involved in the proof of Theorem \ref{thm:1}, in particular giving the definition of \textit{almost harmonic Maass forms}. In Section 3, we describe and recall some results on Jacobi forms. In Section 4, we investigate the properties of a particular function that plays a key role in this work. In Section 5, we recall the decomposition \eqref{eq:2} of a meromorphic Jacobi form in the sense of Dabholkar, Murthy, and Zagier. Furthermore, we describe how to obtain a decomposition as in \eqref{eq:3} for any meromorphic Jacobi form, i.e., how to complete the two functions that arise in the previous decomposition in order to obtain a new decomposition whose summands transform well. Finally, we prove Theorem \ref{thm:1} in Section 6.

\section{Almost harmonic Maass forms}

In this section we give the definition of new automorphic objects recently introduced by Bringmann--Folsom \cite{BF2}, called \textit{almost harmonic Maass forms}. In order to do that we also recall other automorphic objects including harmonic weak Maass forms and almost holomorphic modular forms.

For a fixed $k \in \frac{1}{2}\Z$, we define 
\begin{equation*}
\Gamma:= \begin{cases}
\SL_{2}(\Z) & \text{if $k\in \Z$,} \\
\Gamma_0 (4) & \text{if $k\notin \Z$.}
\end{cases}
\end{equation*}
Let $ f: \mathbb{H} \rightarrow \mathbb{C} $ be a smooth function. 
 For each $\gamma= 
\bigl(
\begin{smallmatrix}
    a & b  \\
    c & d
  \end{smallmatrix}
\bigr) \in \Gamma$  we define the weight $k\in \frac{1}{2}\Z$ \textit{slash--operator} $|_k$ by
\[
f|_k \gamma (\tau):=j(\gamma,\tau)^{-2k} f(\gamma \tau),
\]
where as usual 
\begin{equation*}
j(\gamma,\tau):= \begin{cases}
\sqrt{c\tau+d} & \text{if $k\in \Z$,} \\
\sqrt{c\tau+d} \epsilon_d ^{-1}\left( \frac{c}{d} \right)  & \text{if $k\notin \Z$,}
\end{cases}
\end{equation*}
with 
\begin{equation*}
\epsilon_d:= \begin{cases}
1 & \text{if $d \equiv 1 \pmod{4}$,} \\
i & \text{if $d \equiv 3 \pmod{4}$},
\end{cases}
\end{equation*}
and where $\left( \frac{c}{d} \right)$ denotes the Jacobi symbol.
Throughout for $\tau \in \mathbb{H}$, we let $\tau=u+iv$, where $u,v \in \mathbb{R}$, and set $q:=e^{2\pi i \tau}$.
The weight $k$ \textit{hyperbolic Laplacian} is defined by 
\[
\Delta _k:=-v^2\left( \frac{\partial ^2}{\partial u^2}+\frac{\partial ^2}{\partial v^2} \right)+ikv \left( \frac{\partial }{\partial u}+i\frac{\partial }{\partial v} \right)  .
\]
\begin{definition}
A smooth function $f:\mathbb{H}\rightarrow \mathbb{C}$ is called a weak Maass form of weight $k\in \frac{1}{2}\Z$, Dirichlet character $\chi$ and Laplace eigenvalue $s \in \mathbb{C}$ on a congruence subgroup $\widetilde{\Gamma}$ of $\Gamma$ if the following conditions hold:
\begin{enumerate}
\item For all $\gamma=\bigl(
\begin{smallmatrix}
    a & b  \\
    c & d
  \end{smallmatrix}
\bigr) \in \widetilde{\Gamma}$
\[
f|_k \gamma ( \tau ) = \chi(d) f(\tau).
\]
\item The function $f$ is an eigenfunction for $\Delta_k$ with eigenvalue $s$, i.e., 
\[
\Delta_k(f)=sf.
\]
\item The function $f$ has at most linear exponential growth at the cusps.
\end{enumerate}
Furthermore, if $s=0$ we call $f$ a harmonic weak Maass form.
\end{definition}

It is well known (see for example Section 3 in \cite{MR2097357}) that a harmonic weak Maass form $f$ canonically splits into a holomorphic part, called a \textit{mock modular form}, and a non-holomorphic part, which can be written as a $q$-series in terms of the incomplete Gamma-function. More precisely let 
\[
\Gamma(\alpha;x):=\int\limits_{x}^{\infty} e^{-t}t^{\alpha-1}dt
\]
be the incomplete Gamma-function. The integral converges for $\alpha >0$, and it can be holomorphically continued in $\alpha$ (for $x\neq 0$). If we set
\[
H(t):=e^{-t}\Gamma(1-k,-2t),
\]
extended by continuity according to the incomplete Gamma-function, then for $k\neq 1$
\[
f(\tau)=c_f^{-}v^{1-k}+\sum_{n\gg -\infty} c_f^{+}(n)q^n +\sum_{n\ll 0} c_f^{-}(n)H(2\pi n v)e(nu).
\]
If $k=1$ then the first term must be replaced by $c_f^{-}\log\rl{v}$.


The other ingredient needed to define almost harmonic Maass forms is an automorphic object first defined by Kaneko-Zagier \cite{MR1363056}, the so-called \textit{almost holomorphic modular form}. Such a function is a special kind of non-holomorphic modular form which can be written as a polynomial in $\frac{1}{v}$ with holomorphic coefficients. Here, accordingly as in \cite{BF2}, we consider a slight modification of this definition, allowing weakly holomorphic coefficients (i.e., it is allowed a principal part in the Fourier expansion at the cusps). Special examples, that emphasize the importance of this automorphic object, are the non-holomorphic Eisenstein series $E_2^{*}$, and the non-holomorphic derivative of holomorphic modular forms. Furthermore we call the holomorphic part of an almost holomorphic modular form, i.e., its constant term, a \textit{quasimodular form}.

Finally, for each $k\in \frac{1}{2}\Z$, we recall the \textit{Maass raising operator} of weight $k$ defined by
\[
R_k := 2i \frac{\partial}{\partial \tau}+\frac{k}{v}.
\]
This operator acts on the space of Maass forms, and in particular it sends a Maass form of weight $k$ and Laplace eigenvalue $s$ to a Maass form of weight $k+2$ and Laplace eigenvalue $s+k$. Therefore, it is natural to define the iterations of $R_k$ for any positive integer $n$ as
\[
R_k^n := R_{k+2(n-1)}\circ \cdots \circ R_{k+2} \circ R_k .
\]
We also denote with $R^0_k$ the identity operator.
\begin{definition} \label{defi:6}
A smooth function $f:\mathbb{H}\rightarrow\mathbb{C}$ is called an almost harmonic Maass form of weight $k\in \frac{1}{2}\Z$ and depth $r\in \mathbb{N}\cup \{0\}$ for a congruence subgroup $\widetilde{\Gamma}$ of $\Gamma$ and character $\chi$ if the following conditions hold:
\begin{enumerate}
\item For all $\gamma=\bigl(
\begin{smallmatrix}
    a & b  \\
    c & d
  \end{smallmatrix}
\bigr) \in \widetilde{\Gamma}$
\[
f|_k \gamma ( \tau ) = \chi(d) f(\tau).
\]
\item The function $f$ can be written as a finite linear combination of objects of the form
\begin{equation} \label{eq:4}
\sum\limits_{j=1}^{r_i} g_{j,i} R_{k+2-\nu_i}^{j-1}\left(h_i\right),
\end{equation}
where the $h_i$'s are harmonic weak Maass forms of weight $k+2-\nu_i$, $\nu_i \in \frac{1}{2}\Z$ are fixed, and $g_{j,i}$ are almost holomorphic modular forms of weight $\nu_i -2j$. 
The depth $r$ of  $f$ is defined as the maximum power of the raising operator, i.e, $r:=max_{i}\{ r_i \}-1$.
\end{enumerate}
\end{definition}
\begin{remark}
We modify the original definition in \cite{BF2} to also allow sums of objects of the form \eqref{eq:4}. 
\end{remark}
\begin{remark}
Each summand in \eqref{eq:4} is an automorphic object for some congruence subgroup of $\widetilde{\Gamma}$. 
\end{remark}
\begin{remark}
For each automorphic object it is possible to generalize the definition allowing it to be \textit{vector-valued}. With that we mean that, for a positive integer $N$, the vector of functions $\textbf{f}(\tau):=(f_j(\tau))_{j\pmod{N}}$ satisfies 
\[
\textbf{f}(\gamma\tau)=(c\tau+d)^k U(\gamma) \textbf{f}(\tau),
\]
for all $\gamma= 
\bigl(
\begin{smallmatrix}
    a & b  \\
    c & d
  \end{smallmatrix}
\bigr)\in \SL_2(\Z)$, where $U(\gamma)$ is a certain $N\times N$ matrix.
\end{remark}
\section{Jacobi forms}
\subsection{Holomorphic and mock Jacobi forms}
Here we recall the basic definitions and properties of holomorphic and mock Jacobi forms. For an extensive description of the holomorphic case we refer to \cite{MR781735}, while for a more general theory of mock Jacobi forms see for instance \cite{MR2680205}. Throughout, for $z\in \mathbb{C}$, we let $z=x+iy$, where $x,y \in \mathbb{R}$, and define $\zeta:= e^{2\pi i z}$. Furthermore, for any $w\in \mathbb{C}$, we denote $e(w):=e^{2\pi i w}$.
\begin{definition} \label{defi:5}
A function $\phi :\mathbb{C}\times \mathbb{H}\to \mathbb{C}$ which is holomorphic in both of the variables is called a holomorphic Jacobi form of weight $k\in \Z$ and index $M \in \mathbb{N}$ for $\SL_2(\Z)$ if the following conditions hold:
\begin{enumerate} 
\item For all $\lambda$, $\mu \in \Z$,
\begin{equation}  \label{eq:5}
\phi \left(z+\lambda \tau +\mu;\tau\right)=e\rl{-M\rl{\lambda ^2 \tau +2\lambda z}}\phi\left(z;\tau\right).
\end{equation}
\item For all $\gamma= 
\bigl(
\begin{smallmatrix}
    a & b  \\
    c & d
  \end{smallmatrix}
\bigr) \in \SL_2(\Z)$,
\begin{equation}  
\phi \rl{\frac{z}{c\tau +d};\gamma \tau}=(c\tau +d)^k e\rlf{Mcz^2}{c\tau +d}\phi \rl{z;\tau}.
\end{equation}
\item The function $\phi$ has a Fourier development of the form 
\[
\sum_{n\geq \frac{r^2}{4M}}c(n,r)q^n \zeta^r .
\]
\end{enumerate}
\end{definition}
\begin{remark}
One can extend the definition in the usual way to include Jacobi forms  with multipliers, as well as half-integral weight, for congruence subgroups, or vector-valued Jacobi forms.
\end{remark}
As mentioned in the introduction, one of the most interesting properties of a holomorphic Jacobi form is that it can be written as 
\begin{equation} \label{eq:7}
\varphi(z;\tau)=\sum_{\ell \pmod{2M}}h_{\ell}(\tau)\vartheta_{M,\ell}(z;\tau),
\end{equation}
where 
\begin{equation} \label{eq:8}
\theta_{M,\ell}(z;\tau):=\sum_{\substack{\lambda \in \Z \\  \lambda \equiv \ell \pmod{2M}}}q^{\frac{\lambda^2}{4M}}\zeta^{\lambda}, 
\end{equation}
and 
\begin{equation} \label{eq:9}
h_{\ell}(\tau):= q^{-\frac{\ell^2}{4M}}\int\limits_{0}^{1}\phi(z;\tau)\zeta^{-\ell}dz
\end{equation}
is the $\ell$-th component of a vector-valued modular form. Notice that this decomposition essentially follows from the fact that $\phi$ is holomorphic and satisfies the elliptic transformation law, which implies that $h_{\ell}$ only depend on $\ell$ modulo $2M$. The modularity of $h_\ell$ comes from the modularity of both $\phi$ and $\theta_{M,\ell}$.

One can also consider holomorphic functions as in Definition \ref{defi:5} that satisfy the first and the third properties, but no longer the modular one. Instead, they satisfy a weaker modular condition that we next explain. From \eqref{eq:5} and from the holomorphicity it follows that $\phi$ still allows a theta decomposition as in \eqref{eq:7}. We call $\phi$ a \textit{mock Jacobi form} if its Fourier coefficients $h_{\ell}$ are mock modular forms. In particular one can consider its (non-holomorphic) completion defined as a theta decomposition, where the mock modular forms $h_{\ell}$ are replaced by their associated harmonic weak Maass forms $\widehat{h}_{\ell}$. In this way the function
\[
\widehat{\phi}(z;\tau):=\sum_{\ell \pmod{2M}}\widehat{h}_{\ell}(\tau)\theta_{M,\ell}(z;\tau)
\]
satisfies both the transformation laws of a Jacobi form.

Important examples of mock Jacobi forms come from the so-called Appell functions, well described by Zwegers \cite{Zwegers1}, \cite{Zwegers:PhD}. Here we recall two of such important examples needed in our work.

For all positive integers $n$ we define the \textit{multivariable Appell function} by
\begin{equation} 
\mu_n (u,\textbf{v};\tau):=\frac{e^{\pi i u}}{\prod_{j=1}^{n}\theta(v_j;\tau)}
\sum_{\textbf{k}\in \Z^n}\frac{(-1)^{|\textbf{k}|} q^{\frac{1}{2}\| \textbf{k} \|^2+\frac{1}{2}| \textbf{k} |}   e\rl{\textbf{k} \cdot \textbf{v}} }{1-e(u)q^{|\textbf{k}|}} ,
\end{equation}
where $u\in \mathbb{C}\smallsetminus (\Z\tau+\Z)$, $\textbf{v}=(v_j)_{j}\in \rl{\mathbb{C}^{\times}}^n$, $\tau \in \mathbb{H}$, and
where we denote $|\textbf{k}|:=\sum_{j=1}^{n}k_j$ and $\|\textbf{k}\|^2:=\sum_{j=1}^{n}k_j ^2$.
Furthermore, for all $0\leq \ell \leq n-1$ we define the \textit{shifted multivariable Appell functions} by
\begin{equation} \label{eq:18}
\mu_{n,\ell}\rl{u,\textbf{v};\tau}:=(-1)^{\ell}q^{-\frac{\ell^2}{2n}}e\rl{-\frac{\ell}{n}(u-|\textbf{v}|)}\mu_{n}\rl{u+\ell\tau,\textbf{v};\tau}.
\end{equation}

As described in \cite{Zwegers1}, $\left(  \mu_{n,\ell} \right)_{\ell=0,\dots,n} $ is a vector-valued multivariable mock Jacobi form, whose completion is given by
\[
\widehat{\mu}_{n,\ell}\rl{u,\textbf{v};\tau}:=\mu_{n,\ell}\rl{u,\textbf{v};\tau}-\frac{i}{2}R\left( u-|\textbf{v}| -\frac{n+1}{2};n\tau \right),
\]
where the real-analytic function $R$ is defined by 
\[
R(u;\tau):=\sum_{\nu \in \frac{1}{2}+\Z} \left\lbrace  \sgn(\nu)-E\left(  \left( \nu+\frac{\im(u)}{v} \right)\sqrt{2v}  \right) \right\rbrace (-1)^{\nu-\frac{1}{2}}q^{-\frac{\nu^2}{2}}e(\nu u), 
\]
with $E(t):=2\int\limits_{0}^{t}e^{-\pi u^2}du$.
The following proposition states the modular and elliptic transformation laws for $\widehat{\mu}_{n,\ell}$ as described by Zwegers in \cite{Zwegers1}.
\begin{proposition} \label{pro:8}
Let $u\in \mathbb{C}\smallsetminus (\Z\tau+\Z)$, $\textbf{v}=(v_j)_{j}\in \rl{\mathbb{C}^{\times}}^n$, $\tau \in \mathbb{H}$, and assume $\lambda_1$, $\nu_1 \in \Z$ and $\boldsymbol\lambda_2$, $\boldsymbol\nu_2 \in \Z^n$ such that $\lambda_1-|\boldsymbol\lambda_2|\in n\Z$. The following are true:
\[
\begin{array}{llcl}
(1)&\widehat{\mu}_n(u,\textbf{v};\tau)&=&(-1)^{\lambda_1+|\boldsymbol\lambda_2|+\nu_1+|\boldsymbol\nu_2|}\e^{-\frac{2\pi i}{n}(\lambda_1-|\boldsymbol\lambda_2|)(u-|\textbf{v}|)}q^{-\frac{1}{2n}(\lambda_1-|\boldsymbol\lambda_2|)^2} \\
&&&\times \widehat{\mu}_n(u+\lambda_1\tau+\nu_1,\textbf{v}+\boldsymbol\lambda_2\tau+\boldsymbol\nu_2;\tau), \\
\\
(2)&\widehat{\mu}_{n,\ell}(u,\textbf{v};\tau+1)&=&\e^{-\frac{\pi i}{n}\rl{\ell -\frac{n}{2}}}\widehat{\mu}_{n,\ell}(u,\textbf{v};\tau), \\
\\
 (3)&\widehat{\mu}_{n,\ell}\rl{\frac{u}{\tau},\frac{\textbf{v}}{\tau};-\frac{1}{\tau}}&=&i^{n+1}\sqrt{\frac{-i\tau}{n\tau}}\e^{-\frac{\pi i}{n\tau}(u-|\textbf{v}|)^2}\sum_{r \pmod{n}}\e^{\frac{2\pi i r\ell}{n}} \widehat{\mu}_{n,r}\rl{u,\textbf{v};\tau}.
\end{array}
\]
\end{proposition}
The second function we need to describe is 
\begin{equation} \label{eq:17}
f^{(M)}_{w}(z;\tau):=\sum_{\alpha \in \Z}(-1)^{2M\alpha} \frac{q^{M\alpha^2}e(2M\alpha z)}{1-e(z-w)q^{\alpha}},
\end{equation}
where $M\in \frac{1}{2}\Z$, and $z,w\in \mathbb{C}$ such that $z-w \notin \Z\tau+\Z$.
Defining for all $\ell \in \Z$ the function 
\begin{equation} \label{eq:15}
R_{M,\ell}(w;\tau):=-i e(w(M-\ell))q^{-\frac{(\ell -M)^2}{4M}}R\rl{2Mw-\frac{1}{2}+\tau(\ell -M);2M\tau},
\end{equation}
it turns out that its completion
\begin{equation}  \label{eq:14}
\widehat{f}^{(M)}_{w}(z;\tau):=f^{(M)}_{w}(z;\tau)-\frac{1}{2}\sum_{\ell \pmod{2M}}R_{M,\ell}(w;\tau)\theta_{M,\ell}(z;\tau)
\end{equation}
transforms like a $2$--variable Jacobi form of index $\bigl(
\begin{smallmatrix}
    M & 0  \\
    0 & -M
  \end{smallmatrix}
\bigr)$ and weight $1$ for $\SL_{2}(\Z)$ \cite{Zwegers:PhD}.
\begin{proposition} \label{pro:6}
The function $\widehat{f}^{(M)}_{w}(z;\tau)$ satisfies the following transformation properties:
\begin{enumerate}
\item For all $\gamma=\bigl(
\begin{smallmatrix}
    a & b  \\
    c & d
  \end{smallmatrix}
\bigr) \in \SL_2(\Z) $
\[
\widehat{f}_{\frac{w}{c\tau +d}}^{(M)}\rl{\frac{z}{c\tau +d};\gamma\tau}=(c\tau+d)e\rlf{Mc\rl{z^2-w^2}}{c\tau +d}\widehat{f}_{w}^{(M)}\rl{z;\tau}.
\]

\item For all $\lambda$, $\mu \in \Z$ 
\[
\widehat{f}_{w}^{(M)}\rl{z+\lambda \tau +\mu;\tau}=e\rl{-M\rl{\lambda^2\tau+2\lambda z}}\widehat{f}_{w}^{(M)}\rl{z;\tau},
\]
and
\[
\widehat{f}_{w+\lambda \tau +\mu}^{(M)}\rl{z;\tau}=e\rl{M\rl{\lambda^2\tau+2\lambda w}}\widehat{f}_{w}^{(M)}\rl{z;\tau}.
\]
\end{enumerate}
\end{proposition}
\subsection{Meromorphic Jacobi forms} \label{subsec:3}
As stated before, a meromorphic Jacobi form $\phi$ is a function that satisfies both of the transformation properties in Definition \ref{defi:5}, but that is meromorphic in the elliptic variable $z$, and weakly holomorphic in $\tau$, i.e., meromorphic at infinity. Here we fix the notation and describe the basic properties of the poles of $\phi$.

For each fixed $\tau \in \mathbb{H}$ denote by $S^{(\tau)}$ the set of poles of $z\mapsto \phi(z;\tau)$. Notice that this set has a nice symmetric structure. Indeed, from the elliptic transformation property of $\phi$ it follows that each pole in $S^{(\tau)}$ is equivalent to a pole in $S_0^{(\tau)}:=S^{(\tau)}\cap P$ after translating by $\Z\tau +\Z$, where $P:=[0,1)\tau +[0,1)$. Moreover, since $P$ is bounded and $\phi$ is meromorphic, $S_0^{(\tau)}$ is finite. We let also
\begin{equation} \label{eq:19}
\mathbb{S}^{(\tau)}:=\left\{ (\alpha,\beta)\in \Q^2 : \alpha \tau+\beta\in S^{(\tau)} \right\},
\end{equation}
and for each $s\in \mathbb{S}^{(\tau)}$ denote the relative pole by $z_s(\tau)=z_s\in S^{(\tau)}$. Finally we define $\mathbb{S}_0^{(\tau)}$ by replacing  $S^{(\tau)}$ by $S_0^{(\tau)}$ in \eqref{eq:19}. 

For each $\gamma=\bigl(
\begin{smallmatrix}
    a & b  \\
    c & d
  \end{smallmatrix}
\bigr) \in \SL_2(\Z) $ and for each $s \in \mathbb{S}^{(\gamma \tau)}$, one has the relation
\[
z_s(\gamma\tau)=\frac{z_{s\gamma}(\tau)}{c\tau+d},
\]
which implies
\begin{equation} \label{eq:20}
S^{(\tau)}=(c\tau+d)S^{(\gamma \tau)}
\end{equation}
and
\begin{equation} \label{eq:21}
\mathbb{S}^{(\gamma \tau)}\gamma=\mathbb{S}^{(\tau)}.
\end{equation}
For each Jacobi form of weight $k$ and index $M$ on $\SL_2(\Z)$, and for each $\alpha$ and $\beta\in \Q$, Theorem 1.3 of \cite{MR781735} implies that  the function $q^{m\alpha^2}\phi(\alpha\tau+\beta;\tau)$ is a modular form of weight $k$ on the finite index subgroup
\[
\Gamma_{\alpha,\beta}:=\left\{ \bigl(
\begin{smallmatrix}
    a & b  \\
    c & d
  \end{smallmatrix}
\bigr) \in \SL_2(\Z) : (a-1)\alpha+c\beta,\ b\alpha+(d-1)\beta,\ m\left(c\beta^2-b\alpha^2+(d-a)\alpha\beta \right) \in \Z \right\}
\]
of $\SL_2(\Z)$. Therefore, if we define the finite index subgroup $\Gamma_{\phi}$ of $\SL_2(\Z)$ by
\[
\Gamma_{\phi}:=\bigcap_{(\alpha,\beta)\in \mathbb{S}_0^{(\tau)}}\Gamma_{\alpha,\beta},
\]
then for all $\gamma \in \Gamma_{\phi}$, and for each $s\in \mathbb{S}^{(\tau)}$, $z_{s}(\gamma \tau)\in S^{(\gamma \tau)}$. This fact, together with \eqref{eq:21}  and the modular law of $\phi$, implies that $\mathbb{S}^{(\tau)}$ is $\Gamma_{\phi}$-invariant (for right multiplication). In fact, it is straightforward to prove that for each $\gamma \in \Gamma_{\phi}$, the map
\begin{displaymath}
\xymatrix{
\mathbb{S}_0^{(\tau)} \ar[r] & \mathbb{S}^{(\tau)}\ar[r] & \mathbb{S}_0^{(\tau)} \\
s \ar[r] & s\gamma \ar[r] & s\gamma \pmod{\Z^2} }
\end{displaymath}
is the identity map.

\section{The function $F^{(s)}(\epsilon;\tau)$}
In this section we describe the function $F^{(s)}(\epsilon;\tau)$, which plays a fundamental role in our work for two reasons: firstly, it relates the Laurent coefficients of a meromorphic Jacobi form to some almost holomorphic modular form, whose non-holomorphic parts can be given as a linear combination of the Laurent coefficients themselves. Secondly, it allows us to relate the image of a certain class of functions under the differential operator $\frac{\partial}{\partial \epsilon}$ to the image under the Maass raising operator.

For $s=(\alpha,\beta)\in \Q^2$, $\epsilon\in \mathbb{C}$, and $\tau \in \mathbb{H}$ we define
\[
F^{(s)}(\epsilon;\tau):=e^{\frac{M\pi \epsilon^2}{v}}e(M\alpha\beta+2M\alpha\epsilon)q^{M\alpha^2}.
\]
Notice that $F^{(s)}$ is holomorphic in $\epsilon$ and non-holomorphic in $\tau$.
\begin{remark}
The function $e^{\frac{M\pi \epsilon^2}{v}}=F^{(0,0)}(\epsilon;\tau)$ appears in \cite{BF2}. In fact, $s$ represents an element of $S_0^{(\tau)}$, and the function studied in \cite{BF2} has a unique pole $0\in S_0^{(\tau)}$.
\end{remark}
A straightforward computation gives the following transformation properties for $F^{(s)}$.
\begin{lemma} \label{lem:1}
The function $F^{(s)}\rl{\epsilon;\tau}$ satisfies the following transformation laws:
\begin{enumerate}
\item For all $\gamma=\bigl(
\begin{smallmatrix}
    a & b  \\
    c & d
  \end{smallmatrix}
\bigr) \in \SL_2(\Z)$
\begin{equation*}
F^{(s)}\rl{\frac{\epsilon}{c\tau +d};\gamma\tau}=e\rl{-\frac{cM}{c\tau+d}\rl{z_{s\gamma}+\epsilon}^2}F^{(s\gamma)}\rl{\epsilon;\tau}.
\end{equation*}
\item For all $\lambda$, $\mu \in \Z$
\begin{equation*}
F^{(s+(\lambda,\mu))}\rl{\epsilon;\tau}=e\rl{M\rl{\alpha\mu-\beta\lambda}}
e\rl{M\rl{\lambda^2\tau+2\lambda(z_s+\epsilon)}}F^{(s)}(\epsilon;\tau).
\end{equation*}
\end{enumerate}
\end{lemma}
Let $\phi$ be a meromorphic Jacobi form, and denote by $z_s=\alpha\tau+\beta$ one of its poles, where $s=(\alpha,\beta)\in \Q^2$. Then we define the Laurent coefficients $\widetilde{D}_j^{(s)}$ of $\phi$ relative to $z_s$ by
\[
\phi(\epsilon+z_s;\tau)=\sum\limits_{j=1}^{n_s}\frac{\widetilde{D}_j^{(s)}(\tau)}{(2\pi i \epsilon)^j}+O(1),
\]
where $n_s$ denotes the order of the pole. Furthermore, we define the functions $D^{(s)}_j$ as the Laurent coefficients of $F^{(s)}\phi$ in the elliptic variable, namely
\begin{equation} \label{eq:22}
F^{(s)}(\epsilon;\tau)\phi(\epsilon+z_s;\tau)=\sum\limits_{j=1}^{n_s}\frac{D_j^{(s)}(\tau)}{(2\pi i \epsilon)^j}+O(1).
\end{equation}
\begin{proposition} \label{pro:3}
The functions $D_j^{(s)}(\tau)$, defined in \eqref{eq:22}, are almost holomorphic  modular forms of weight $k-j$ on $\Gamma_{\phi}$. Furthermore, their holomorphic part is uniquely determined by the $\widetilde{D}_{\lambda}^{(s)}(\tau)$'s.
\end{proposition}
\begin{proof}
Firstly, we prove the modularity.  
From the definition of  Jacobi forms and from Lemma \ref{lem:1} it follows that for each $\gamma \in \Gamma_{\phi}$ 
\[
F^{(s)}\rl{\frac{\epsilon}{c\tau+d};\gamma\tau}\phi\rl{\frac{\epsilon}{c\tau+d}+z_s(\gamma\tau);\gamma\tau} =
(c\tau+d)^k F^{(s\gamma)}(\epsilon;\tau)\phi(\epsilon+z_{s\gamma};\tau) .
\]
Using the elliptic transformation properties of both $F^{(s)}$ and $\phi$, we shift $s\gamma$ to  $s$ by virtue of our discussion in Subsection \ref{subsec:3}, say
\[
s=s\gamma +(\lambda,\mu),
\]
for some $(\lambda,\mu)\in \Z^2$, obtaining
\[
(c\tau+d)^k F^{(s\gamma)}(\epsilon;\tau)\phi(\epsilon+z_{s\gamma};\tau) = e(M(\alpha\mu-\beta\lambda))(c\tau+d)^k F^{(r)}(\epsilon;\tau)\phi(\epsilon+z_{r};\tau).
\]
Notice that since $\gamma \in \Gamma_\phi$, the coefficient $e(M(\alpha\mu-\beta\lambda))$ is in fact equal to $1$.
In particular, writing both the right and the left hand sides in terms of the Laurent expansion, we obtain
\[
D_j^{(s)}(\gamma\tau)=(c\tau+d)^{k-j}D_j^{(s)}(\tau),
\]
which proves the modular property. 
It remains to prove that they can be written as polynomials in $\frac{1}{v}$ with weakly holomorphic coefficients. 
Clearly, each of the $D_j^{(s)}(\tau)$ can be written as combinations of the Laurent coefficients of $\phi(\epsilon +z_s;\tau)$ and $F^{(s)}(\epsilon;\tau)$ in $\epsilon=0$. More precisely, it is easy to see that 
\[
D_j^{(s)}(\tau)=\sum_{\lambda=0}^{n_s-j}\frac{1}{(2\pi i)^{\lambda}\lambda!}\widetilde{D}_{\lambda}^{(s)}(\tau)
\frac{\partial^{j-\lambda}}{\partial \epsilon^{j-\lambda}}\ql{F^{(s)}(\epsilon;\tau)}_{\epsilon=0}.
\]
It is straightforward to show that $\frac{\partial^n}{\partial \epsilon^n}\ql{F^{(s)}(\epsilon;\tau)}_{\epsilon=0}$ is equal to $q^{\alpha ^2M}$ times a polynomial in $\frac{1}{v}$ with coefficients in $\mathbb{C}$. Furthermore, its constant term is given by $(4\pi i M\alpha)^{n}$. From these observations it follows that $D_j^{(s)}(\tau)$ is an almost holomorphic modular form, whose holomorphic part is given by
\[
q^{\alpha^2 M}e(M\alpha\beta)\sum_{\lambda=0}^{n_s-j}\widetilde{D}_{\lambda}^{(s)}(\tau)(4\pi i M\alpha)^{j-\lambda}.
\]
\end{proof}

The second and most important property of $F^{(s)}$ states that if a smooth function $g(\epsilon;\tau)$ is annihilated by the level $M$ heat operator
\[
H_M:=8\pi iM\frac{\partial}{\partial \tau}+\frac{\partial^2}{\partial \epsilon^2}
\]
for some integer $M$, then the image of $g$ under the raising operator is related to the image of the ratio of $g$ with $F^{(0,0)}$ under the operator $\frac{\partial}{\partial \epsilon}$. The precise version of this statement is given below. Here and throughout $\delta_{\epsilon}$ denotes the differential $\frac{1}{2\pi i}\frac{\partial}{\partial \epsilon}$.

\begin{proposition} \label{pro:7}
Let $g(\epsilon;\tau)$ be a smooth function in both the variables such that for some $M\in \Z$
\begin{equation} \label{eq:27}
H_M[g(\epsilon;\tau)]=0,
\end{equation}
then for all $j\geq 0$ one has
\begin{equation} \label{eq:25}
\delta_{\epsilon}^{2j}\left[ \frac{g(\epsilon;\tau)}{F^{(0,0)}(\epsilon;\tau)} \right]_{\epsilon=0}
=\rl{\frac{M}{\pi}}^j R_{\frac{1}{2}}^j\rl{g(0;\tau)},
\end{equation}
and 
\begin{equation} \label{eq:16}
\delta_{\epsilon}^{2j+1}\left[ \frac{g(\epsilon;\tau)}{F^{(0,0)}(\epsilon;\tau)} \right]_{\epsilon=0}
=\rl{\frac{M}{\pi}}^j R_{\frac{3}{2}}^j\left(\delta_{\epsilon}\left[ g(\epsilon;\tau)\right]_{\epsilon=0} \right).
\end{equation}
\end{proposition}

\begin{proof}
We only give a sketch of the proof since it is rather technical, and it follows essentially the computation used in Theorem 3.5 in \cite{BF2}. There the authors proved \eqref{eq:16} in the case of 
$g(u;\tau)=R_{M,\ell}(u;\tau)$. Notice that the operator $\mathbb{D}_{\epsilon}$ used in \cite{BF2} is equivalent to the left hand side of \eqref{eq:16}. We point out that in their computation the only property that they need is that $R_{M,\ell}$ is annihilated by the heat operator.

We prove \eqref{eq:25}. Similarly one can prove \eqref{eq:16}.
From \eqref{eq:27} we know that
\[
\delta_{\epsilon}^{2k}\ql{g(\epsilon;\tau)}_{\epsilon=0}=\rl{\frac{2iM}{\pi}}^k \frac{\partial^{k}}{\partial \tau^{k}}\ql{g(0;\tau)}.
\]
It follows that the left hand side of \eqref{eq:25} equals
\begin{multline} \label{eq:26}
\delta_\epsilon ^{2j}  \ql{e^{-\frac{M\epsilon^2}{v}}g\rl{\epsilon;\tau}}_{\epsilon=0}=
\sum_{k=0}^{j}\binom {2j}{2k}\delta_\epsilon ^{2k}\ql{g\rl{\epsilon;\tau}}_{\epsilon=0}\delta_\epsilon ^{2j-2k}\ql{e^{-\frac{ M \epsilon ^2}{v}}}_{\epsilon=0}\\
=\rl{  \frac{M}{\pi} }^j\sum_{k=0}^j \frac{(2j)!}{(2k)!(j-k)!}i^k 2^{3k-2j}v^{k-j} \frac{\partial^k}{\partial\tau^k}\ql{g\rl{0;\tau}}.
\end{multline}
If we set
\[
\alpha _{k,j}:= \frac{(2j)!}{(2k)!(j-k)!}i^k 2^{3k-2j}
\]
and
\[
g_j (\tau):=\sum_{k=0}^j \alpha_{k,j}v^{k-j} \frac{\partial^k}{\partial\tau^k}\ql{g\rl{0;\tau}},
\]
then \eqref{eq:26} can be written as
\[
\delta_\epsilon ^{2j}  \ql{e^{-\frac{M\epsilon^2}{v}}g\rl{\epsilon;\tau}}_{\epsilon=0}=\left(  \frac{M}{\pi} \right)^j g_j (\tau).
\]
A straightforward calculation gives
\[
g_1 (\tau)=R_{\frac{1}{2}}\rl{g(0;\tau)},
\]
and then by induction one can show that
\[
g_{j+1}(\tau)=R_{2j+\frac{1}{2}}\rl{g_j (\tau)},
\]
which proves the proposition.
\end{proof}

\section{Decompositions of a meromorphic Jacobi form}
In this section we recall the canonical decomposition of a meromorphic Jacobi form as made by Dabholkar, Murthy, and Zagier \cite{DMZ}. In their work the modular property of each summand is proved for poles of order at most 2. In the last part of this section we provide a second decomposition into two non-holomorphic Jacobi forms, that allows us to understand the modularity in the general case.

\subsection{The canonical splitting}
In \cite{DMZ}, the authors faced and solved the problem of defining a theta decomposition of meromorphic Jacobi forms. The basic idea was to construct an analogue of the Fourier coefficients $h_{\ell}$ preserving the structure and the fundamental properties. In order to do so, they defined the \textit{canonical Fourier coefficients} of the meromorphic Jacobi form $\phi$ of weight $k$ and index $M$ as
\begin{equation} \label{eq:23}
h_{\ell}(\tau):=q^{-\frac{\ell ^2}{4M}}\int\limits_{-\frac{\ell \tau}{2M}}^{-\frac{\ell \tau}{2M}+1}\phi(z;\tau)\zeta^{-\ell}dz.
\end{equation}
\begin{remark}
In the integral above the path of integration is the straight line if there are no poles on it. If $\phi$ has poles on this line, then $h_{\ell}$ is not well defined. For details about this case refer to \cite{DMZ}.
\end{remark}
The key of this definition is that the canonical Fourier coefficients depend just on $\ell$ modulo $2M$, i.e.,
\[
h_{\ell +2M}(\tau)=h_{\ell}(\tau).
\]
 Note that they are essentially the local Fourier coefficients of $\phi$ in a certain region. This implies the possibility of defining the \textit{finite part} of $\phi$ as a theta expansion, namely
\[
\phi^F(z;\tau):=\sum_{\ell \pmod{2M}} h_{\ell}(\tau)\theta_{M,\ell}(z;\tau).
\]
Notice that if $\phi$ is holomorphic then $\phi^F$ coincides with $\phi$. The main difference in the meromorphic case is that the coefficients $h_{\ell}$ are no longer modular, as well as $\phi^F$. On the other hand the \textit{polar part} $\phi^P:=\phi-\phi^F$ has a nice structure that allows us to understand how far are the canonical Fourier coefficients from being modular. Equivalently we are able to complete $h_{\ell}$ in order to obtain a modular object, in a canonical way.

In the following proposition we describe the structure of $\phi^P$. The result and its proof are analogous to the ones in Theorem 3.3 in \cite{BF2}. In order to state it, we define the Laurent expansion of $\phi$ at each of its poles $z_s$ by
\begin{equation} \label{eq:13}
\phi\rl{\epsilon +z_s ;\tau}=\sum_{j=1}^{n_s} \frac{\widetilde{D}_j ^{(s)}(\tau)}{\rl{2\pi i \epsilon}^j} + O(1).
\end{equation}

While computing $\phi^P$, the mock Jacobi form in \eqref{eq:17} and a theta-like decomposition with coefficients
\begin{equation} \label{eq:12}
\xi^{(\alpha,\beta)}_{M,\ell}(u;\tau):=\sum_{\substack{  r\in \Z  \\ r\equiv \ell \pmod{2M} }}\frac{\sgn\rl{r+\frac{1}{2}}-\sgn\rl{r+2M\alpha}}{2}
q^{-\frac{r^2}{4M}}e(-ru)
\end{equation}
naturally appear.
\begin{proposition} \label{pro:2}
For each pole $z_s \in S_0$ let $\widetilde{D}_j^{(s)}$ as in \eqref{eq:13}, $f^{(M)}$ as in \eqref{eq:17} and $\xi^{(s)}_{M,\ell}$ as in \eqref{eq:12}. Then the polar part of $\phi$ can be written as
\[
\phi^{P}(z;\tau)= -\sum_{z_s \in S_0^{(\tau)}} \sum_{j=1}^{n_s} \frac{\widetilde{D}_j ^{(s)}(\tau)}{(j-1)!} 
\delta_\epsilon^{j-1}
\ql{ f^{(M)}_{z_s+\epsilon}(z;\tau)-\sum_{\ell  \pmod{2M}}\xi^{(s)}_{M,\ell}(\epsilon+z_s;\tau)\theta_{M,\ell}(z;\tau)}_{\epsilon=0}.
\]
\end{proposition}
\begin{proof}
Here we give a sketch of the proof. For details see the analogous proof of Theorem 3.3 in \cite{BF2}.
Let $\widetilde{z} :=A\tau+B \in \mathbb{C}$ be fixed, where $A$, $B\in \Q$. Since both $\phi$ and $\phi^F$ are meromorphic then it is not a restriction to assume $\im(z)=\im(\widetilde{z})=Av$. Therefore, defining $h_\ell^{(\widetilde{z})}$ by
\[
h_{\ell}^{(\widetilde{z})}(\tau):=q^{-\frac{\ell^2}{4M}}\int_{\widetilde{z}}^{\widetilde{z}+1}\phi\rl{z;\tau}\zeta^{-\ell}dz,
\]
we can write
\begin{equation} \label{eq:28}
\phi^P\rl{z;\tau}=\phi\rl{z;\tau}-\phi^F\rl{z;\tau}=\sum_{\ell \in \Z}\rl{h_{\ell}^{(\widetilde{z})}(\tau)-h_{\ell}(\tau)}q^{\frac{\ell^2}{4M}}\zeta^{\ell}.
\end{equation}
Let $P_{\widetilde{z}}$ be the parallelogram of vertices $\{ \widetilde{z}, \ \widetilde{z}+1, \ -\frac{\ell \tau}{2M}, \ -\frac{\ell \tau}{2M}+1  \}$, and 
\[
S_{(\tau, \widetilde{z})}:=P_{\widetilde{z}} \cap S^{(\tau)},
\]
where as before $S^{(\tau)}$ is the set of poles of $z\mapsto \phi\rl{z;\tau}$. Applying the residue theorem to \eqref{eq:28} we obtain
\begin{equation} \label{eq:29}
\phi^{P}\rl{z;\tau}=2\pi i \sum_{\ell \in \Z}\sum_{z_s \in S_{(\tau, \widetilde{z})}}\Res_{z=z_s}\left[  \phi\rl{z;\tau}\zeta^{-\ell} \right]\zeta^{\ell}.
\end{equation}
Since any pole in $S_{(\tau, \widetilde{z})}$ is equivalent to a pole in $S^{(\tau)}_0$ modulo $\Z \tau +\Z$, then writing the residue as
\[
\Res_{\epsilon=0}\left[  \phi\rl{z_s+\epsilon;\tau}\e\rl{-\ell(\epsilon+z_s)} \right]=\sum_{j=1}^{n_s}\frac{\widetilde{D}_j^{(s)}(\tau)}{2\pi i (j-1)!}\delta_{\epsilon}^{j-1}\left[  \e\rl{-\ell(\epsilon+z_s)} \right]_{\epsilon=0}, 
\]
and using the elliptic transformation law of $\phi$, we obtain
\begin{multline*} 
\phi^P (z;\tau)= -\sum_{z_s \in S_0^{(\tau)}} \sum_{j=1}^{n_s} \frac{\widetilde{D}_j ^{(s)}(\tau)}{(j-1)!} 
\delta_\epsilon^{j-1}
\ql{ \sum_{\lambda \in \Z}      q^{M\lambda ^2}   e\rl{2M\lambda z}   \right. \\
\left. \times \sum_{\ell \in \Z}          
\frac{\sgn\rl{-\alpha  +\lambda+A}+\sgn\rl{\ell +2M\alpha  }}{2} 
q^{\ell \lambda} e\rl{\ell \rl{z-\epsilon -z_s}}}_{\epsilon=0},
\end{multline*}
where $z_s=(\alpha,\beta)$. To conclude the proof it is enough to notice that the argument of the differential operator is in fact 
\[
f^{(M)}_{z_s+\epsilon}(z;\tau)-\sum_{\ell  \pmod{2M}}\xi^{(s)}_{M,\ell}(\epsilon+z_s;\tau)\theta_{M,\ell}(z;\tau).
\]
\end{proof}

\subsection{The completed splitting}
The splitting showed in the previous subsection does not differ from the one in \cite{DMZ}. So far the only difference is the way we express the polar part. Using this decomposition as a starting point, we provide a second decomposition adding and subtracting a non holomorphic term to $\phi^F$ and $\phi^P$ respectively, in order to obtain the two modular completions $\widehat{\phi}^F$ and $\widehat{\phi}^P$.

Analysing the representation of $\phi^P$ in Proposition \ref{pro:2} it is natural to define its completion by
\begin{equation} \label{eq:24}
\widehat{\phi}^{P}(z;\tau):= -\sum_{z_s \in S_0^{(\tau)}} \sum_{j=1}^{n_s} \frac{D_j ^{(s)}(\tau)}{(j-1)!} 
\delta_\epsilon^{j-1}
\ql{ \frac{\widehat{f}^{(M)}_{z_s+\epsilon}(z;\tau)}{F^{(s)}(\epsilon;\tau)}}_{\epsilon=0},
\end{equation}
where $D_j^{(s)}(\tau)$ are the almost holomorphic modular forms associated to the Laurent coefficients  $\widetilde{D}_j^{(s)}(\tau)$ of $\phi(\epsilon +z_s;\tau)$, as described in Proposition \ref{pro:3}, and $F^{(s)}$ is the function described in Section $3$.
\begin{proposition} \label{pro:4}
The completion of the polar part \eqref{eq:24} can be written as
\begin{multline*}
\widehat{\phi}^P(z;\tau)=\phi^P(z;\tau)+\sum_{z_s \in S_0^{(\tau)}} \sum_{j=1}^{n_s} \frac{D_j ^{(s)}(\tau)}{(j-1)!} 
\\ \times \sum_{\ell \pmod{2M}}\delta_\epsilon^{j-1}
\ql{\frac{\frac{1}{2}R_{M,\ell}(\epsilon+z_s;\tau)-\xi^{(s)}_{M,\ell}(\epsilon+z_s;\tau)}{F^{(s)}(\epsilon;\tau)}}_{\epsilon=0}\theta_{M,\ell}(z;\tau).
\end{multline*}
\end{proposition}
\begin{proof}
Along the proof of Proposition \ref{pro:2} one can see that it is also possible to write $\phi^P$ in terms of $D_{2j}$ instead of $\widetilde{D}_{2j}$. More precisely, one can divide and multiply the argument of the residue in \eqref{eq:29} by $F^{(s)}(\epsilon;\tau)$. Considering now the Laurent expansion of $F^{(s)}(\epsilon;\tau)\phi(\epsilon+z_s;\tau)$ instead of $\phi(\epsilon+z_s;\tau)$ and proceeding with the same computation as is Proposition $\ref{pro:2}$ one obtains
\[
\phi^{P}(z;\tau)= -\sum_{z_s \in S_0^{(\tau)}} \sum_{j=1}^{n_s} \frac{D_j ^{(s)}(\tau)}{(j-1)!} 
\delta_\epsilon^{j-1}
\ql{\frac{ f^{(M)}_{z_s+\epsilon}(z;\tau)-\sum_{\ell  \pmod{2M}}\xi^{(s)}_{M,\ell}(\epsilon+z_s;\tau)\theta_{M,\ell}(z;\tau)}{F^{(s)}(\epsilon;\tau)}}_{\epsilon=0}.
\]
It is enough to use \eqref{eq:14} to conclude the proof.
\end{proof}

Using Lemma \ref{lem:1} and Proposition \ref{pro:4} we define the completion of the finite part as 
\[
\widehat{\phi}^F(z;\tau):=\sum_{\ell \pmod{2M}} \widehat{h}_{\ell}(\tau)\theta_{M,\ell}(z;\tau),
\]
where the functions $\widehat{h}_{\ell}$ are defined as the \textit{completion of the canonical Fourier coefficients} $h_{\ell}$:
\[
\widehat{h}_{\ell}(\tau):=h_{\ell}(\tau)-\sum_{z_s \in S_0^{(\tau)}} \sum_{j=1}^{n_s} \frac{D_j ^{(s)}(\tau)}{(j-1)!} \delta_\epsilon^{j-1}
\ql{\frac{\frac{1}{2}R_{M,\ell}(\epsilon+z_s;\tau)-\xi^{(s)}_{M,\ell}(\epsilon+z_s;\tau)}{F^{(s)}(\epsilon;\tau)}}_{\epsilon=0}.
\]
\begin{proposition}\label{pro:5}
The functions $\widehat{\phi}^F$ and $\widehat{\phi}^P$ satisfy the same transformation properties as $\phi$. Furthermore $\rl{\widehat{h}_\ell (\tau)}_{\ell \pmod{2M}}$ transforms as a vector-valued modular form of weight $k-\frac{1}{2}$ for $\SL_2(\Z)$.
\end{proposition}

\begin{proof}
The elliptic transformation property follows from the analogous transformation for
$\widehat{f}^{(M)}_{z_s+\epsilon}(z;\tau)$. In order to show the modularity property, for all $\gamma=\bigl(
\begin{smallmatrix}
    a & b  \\
    c & d
  \end{smallmatrix}
\bigr) \in \SL_2(\Z)$ we now consider

\[
\widehat{\phi}^{P}\rl{\frac{z}{c\tau+d};\gamma\tau}= -\sum_{z_s(\gamma\tau) \in S_0^{(\gamma\tau)}} \sum_{j=1}^{n_s} \frac{D_j ^{(s)}(\gamma\tau)}{(j-1)!} 
\delta_{\frac{\epsilon}{c\tau+d}}^{j-1}
\ql{\frac{\widehat{f}^{(M)}_{\frac{z_{s\gamma}(\tau)+\epsilon}{c\tau+d}}(\frac{z}{c\tau+d};\gamma\tau)}{F^{(s)}\rl{\frac{\epsilon}{c\tau+d};\gamma\tau}} }_{\epsilon=0}.
\]
Using Lemma \ref{lem:1}, Proposition \ref{pro:3}, and the transformation properties of 
$\widehat{f}^{(M)}_{z_s+\epsilon}(z;\tau)$ in Proposition \ref{pro:6}, we can write it as
\[
-(c\tau+d)^k\e\rl{\frac{cMz^2}{c\tau+d}}\sum_{z_{s\gamma}(\tau) \in S_0^{(\tau)}} \sum_{j=1}^{n_s} \frac{D_j ^{(s\gamma)}(\tau)}{(j-1)!} 
\delta_\epsilon^{j-1}
\ql{\frac{\widehat{f}^{(M)}_{z_{s\gamma}+\epsilon}(z;\tau)}{F^{(s\gamma)}(\epsilon;\tau)} }_{\epsilon=0}.
\]
Notice that the sum over $z_s(\gamma\tau)\in S_0^{(\gamma\tau)}$ is the same as the sum over $z_{s\gamma}(\tau)\in S_0^{(\tau)}$ by virtue of \eqref{eq:20}.

Finally, the modularity of $\rl{\widehat{h}_\ell (\tau)}_{\ell \pmod{2M}}$ follows since its components are the theta-coefficients of $\widehat{\phi}^F(z;\tau)$, which transforms as a standard Jacobi form.
\end{proof}

\section{Proof of the main Theorem}
In this section we prove Theorem \ref{thm:1}, which follows from Theorem \ref{thm:2} below.
\begin{theorem}\label{thm:2}
The functions $\widehat{h}_{\ell}(\tau)$ are (vector--valued) almost harmonic Maass forms of weight $k-\frac{1}{2}$ and depth $\ql{\frac{N-1}{2}}$ for $\SL_2(\Z)$, where $N$ is the highest order of the poles of $\phi$.
\end{theorem}
In order to prove Theorem \ref{thm:2} we need the following result, that allows us to write $\widehat{h}_{\ell}$ in terms of the raising operator.
\begin{lemma} \label{lem:2}
Let $R_{M,\ell}$ as in \eqref{eq:15}  and $\xi_{M,\ell}^{(s)}$ as in \eqref{eq:12}, then for all $s=(\alpha,\beta)\in \Q^2$ we have
\begin{multline*}
\delta_\epsilon ^{2j}\ql{\frac{\frac{1}{2}R_{M,\ell}(\epsilon+z_s;\tau)-\xi^{(s)}_{M,\ell}(\epsilon+z_s;\tau)}{F^{(s)}(\epsilon;\tau)}}_{\epsilon=0} \\
=\rlf{M}{\pi}^j R_{\frac{1}{2}}^{j}\rl{e\rl{-m\alpha\beta}q^{-M\alpha^2 }\rl{\frac{1}{2}R_{M,\ell}(z_s;\tau)-\xi^{(s)}_{M,\ell}(z_s;\tau)}}, 
\end{multline*}
and
\begin{multline*}
\delta_\epsilon ^{2j+1}\ql{\frac{\frac{1}{2}R_{M,\ell}(\epsilon+z_s;\tau)-\xi^{(s)}_{M,\ell}(\epsilon+z_s;\tau)}{F^{(s)}(\epsilon;\tau)}}_{\epsilon=0} \\
=\rlf{M}{\pi}^j R_{\frac{3}{2}}^{j}\rl{\delta_\epsilon \ql{e\rl{-m\alpha\beta-2M\alpha\epsilon}q^{-M\alpha^2 } \rl{\frac{1}{2}R_{M,\ell}(\epsilon+z_s;\tau)-\xi^{(s)}_{M,\ell}(\epsilon+ z_s;\tau)}}_{\epsilon=0}}.
\end{multline*}
\end{lemma}
\begin{proof}
According to Proposition \ref{pro:7} it is enough to prove that 
\[
H_M\ql{e(-2M\alpha\epsilon)q^{-M\alpha^2}\rl{\frac{1}{2}R_{M,\ell}(\epsilon+z_s;\tau)-\xi^{(s)}_{M,\ell}(\epsilon+ z_s;\tau)}}=0 .
\]
By virtue of \cite{MR2661165} the first term is annihilated, namely
\[
H_M\ql{e(-2M\alpha\epsilon)q^{-M\alpha^2}R_{M,\ell}(\epsilon+z_s;\tau)}=0,
\]
while for the other term it is just a straightforward computation.
\end{proof}

Finally, we need to define a vector-valued harmonic weak Maass form, whose non-holomorphic part is essentially
\[
\frac{1}{2}R_{M,\ell}(\epsilon+z_s;\tau)-\xi^{(s)}_{M,\ell}(\epsilon+ z_s;\tau).
\]
In order to do so, for $s=(\alpha,\beta)\in \Q^2\smallsetminus\{(0,0)\}$ we consider the function
\begin{multline*}
\rho^{(s)}_{M,\ell}(u;\tau):=e(-M\alpha\beta)e\rl{-2M\alpha u}
q^{-M\alpha^2 } \\
\times \rl{\mu_{2M,\ell-M}\rl{2M(u+\alpha\tau+\beta),\textbf{v};\tau} +\xi^{(s)}_{M,\ell}(u+\alpha \tau+\beta;\tau)},
\end{multline*} 
where $\textbf{v}:=\left( -\frac{1}{2},\frac{1}{2},\dots ,-\frac{1}{2},\frac{1}{2} \right) $, $\xi^{(s)}_{M,\ell}$ is the function defined in \eqref{eq:12}, and $\mu_{2M,\ell-M}$ is the Appell function defined in \eqref{eq:18}.

Similarly, for $s=(0,0)$ we set
\[
\rho^{(0,0)}_{M,\ell}(u;\tau):=\mu_{2M,\ell-M}\rl{2Mu+\frac{\tau}{2},\textbf{w};\tau} +\xi^{(0,0)}_{M,\ell}(u;\tau),
\]
where $\textbf{w}:=\textbf{v}+\left( \frac{\tau}{2},0,0,\dots ,0 \right) $. 
Finally, we define their completions by
\[
\widehat{\rho}^{(s)}_{M,\ell}(u;\tau):=e(-M\alpha\beta)e\rl{-2M\alpha u}
q^{-M\alpha^2 }\widehat{\mu}_{2M,\ell-M}\rl{2M(u+\alpha\tau+\beta),\textbf{v};\tau}
\]
and
\begin{equation}
\widehat{\rho}^{(0,0)}_{M,\ell}(u;\tau):=\widehat{\mu}_{2M,\ell-M}\rl{2Mu+\frac{\tau}{2},\textbf{w};\tau},
\end{equation}
respectively.

\begin{remark} \label{rem:1}
From \cite{Zwegers1} it follows that 
\[
\widehat{\rho}^{(s)}_{M,\ell}(u;\tau)=\rho^{(s)}_{M,\ell}(u;\tau) +e(-M\alpha\beta-2M\alpha\epsilon)q^{-\alpha^2M}\rl{\frac{1}{2}R_{M,\ell}(\epsilon+z_s;\tau)-\xi^{(s)}_{M,\ell}(\epsilon+z_s;\tau)},
\]
where $R_{M,\ell}$ is defined in \eqref{eq:15}.
\end{remark}
From the following lemma and from the transformation properties of $\widehat{\mu}_{\ell,2M}$ we construct two (vector-valued) harmonic weak Maass forms, for each finite set $S\in \Q^2$ preserved by the action of $\widetilde{\Gamma}$ modulo $\Z^2$, where $\widetilde{\Gamma}$ is a congruence subgroup of $\SL_2(\Z)$. With this we simply mean that for each $s \in S$ and for each $\gamma \in \widetilde{\Gamma}$, there exists $\lambda \in \Z^2$ such that $s= s \gamma+\lambda$. The following lemma follows from a straightforward calculation using the elliptic property of $\widehat{\mu}_{2M,\ell}$ (see Proposition \ref{pro:8}).

\begin{lemma} \label{lem:3}
If $s \in \Q ^2$, then for all $(\lambda,\mu) \in \Z^2$
\[
\widehat{\rho}^{(s+(\lambda,\mu))}_{M,\ell}(u;\tau) =e\rl{M(-\alpha\mu+\beta\lambda)}\widehat{\rho}^{(s)}_{M,\ell}(u;\tau) .
\]
\end{lemma}
By a direct calculation, from Lemma \ref{lem:3} and from Proposition \ref{pro:8}, we obtain  the following result.
\begin{proposition} \label{pro:1}
If $S\in \Q^2$ is a finite set preserved by the action of a congruence subgroup $\widetilde{\Gamma}$ of $\SL_2(\Z)$ modulo $\Z^2$, then 
the functions $\rl{\widehat{\rho}^{(s)}_{M,\ell}(0;\tau)}_{\ell \pmod{2M}}$ and $\rl{\frac{\partial}{\partial \epsilon} \ql{\widehat{\rho}^{(s)}_{M,\ell}(\epsilon;\tau)}_{\epsilon=0}}_{\ell \pmod{2M}}$ are vector-valued harmonic weak Maass forms of weights $\frac{1}{2}$ and $\frac{3}{2}$, respectively, for $\widetilde{\Gamma} \cap \Gamma(2)$.
\end{proposition}

\begin{proof}[Proof of Theorem \ref{thm:2}]
We have already seen in Proposition \ref{pro:5} that $\widehat{h}_{\ell}(\tau)$ transforms like a (vector-valued) modular form of weight $k-\frac{1}{2}$ for $\SL_2(\Z)$. Furthermore, according to Lemma
\ref{lem:2} and Remark \ref{rem:1}, we can rewrite it in terms of the raising operator, namely
\begin{multline*}
\widehat{h}_{\ell}(\tau)=h_\ell(\tau)
+\sum_{z_s(\tau) \in S_0^{(\tau)}} \sum_{h=0}^{\ql{\frac{n_s-1}{2}}} \frac{D_{2h+1} ^{(s)}(\tau)}{(2h)!}\rlf{M}{\pi}^h 
R^{h}_{\frac{1}{2}} \rl{\rho^{(s)}_{M,\ell}(0;\tau)}\\
+\sum_{z_s(\tau) \in S_0^{(\tau)}} \sum_{h=0}^{\ql{\frac{n_s}{2}}-1} \frac{D_{2h+2} ^{(s)}(\tau)}{(2h+1)!}\rlf{M}{\pi}^h 
R^{h}_{\frac{3}{2}} \rl{\delta_\epsilon\ql{\rho^{(s)}_{M,\ell}(\epsilon;\tau)}_{\epsilon=0}}\\
-\sum_{z_s(\tau) \in S_0^{(\tau)}} \sum_{h=0}^{\ql{\frac{n_s-1}{2}}} \frac{D_{2h+1} ^{(s)}(\tau)}{(2h)!}\rlf{M}{\pi}^h 
R^{h}_{\frac{1}{2}} \rl{\widehat{\rho}^{(s)}_{M,\ell}(0;\tau)}\\
-\sum_{z_s(\tau) \in S_0^{(\tau)}} \sum_{h=0}^{\ql{\frac{n_s}{2}}-1} \frac{D_{2h+2} ^{(s)}(\tau)}{(2h+1)!}\rlf{M}{\pi}^h 
R^{h}_{\frac{3}{2}} \rl{\delta_\epsilon\ql{\widehat{\rho}^{(s)}_{M,\ell}(\epsilon;\tau)}_{\epsilon=0}}.
\end{multline*}
Notice that Proposition \ref{pro:1} and Proposition \ref{pro:3} tell us that the last two sums are (vector-valued) almost harmonic Maass forms of weight $k-\frac{1}{2}$. In fact, a computation similar to the proof of Proposition \ref{pro:5} tells us that these objects transform with respect to $\Gamma(2)$. Therefore, the only thing that remains to prove is that the function
\begin{multline*}
h_\ell(\tau)
+\sum_{z_s(\tau) \in S_0^{(\tau)}} \sum_{h=0}^{\ql{\frac{n_s-1}{2}}} \frac{D_{2h+1} ^{(s)}(\tau)}{(2h)!}\rlf{M}{\pi}^h 
R^{h}_{\frac{1}{2}} \rl{\rho^{(s)}_{M,\ell}(0;\tau)}\\
+\sum_{z_s(\tau) \in S_0^{(\tau)}} \sum_{h=0}^{\ql{\frac{n_s}{2}}-1} \frac{D_{2h+2} ^{(s)}(\tau)}{(2h+1)!}\rlf{M}{\pi}^h 
R^{h}_{\frac{3}{2}} \rl{\delta_\epsilon\ql{\rho^{(s)}_{M,\ell}(\epsilon;\tau)}_{\epsilon=0}}
\end{multline*}
is the $\ell$-th component of a vector-valued almost holomorphic modular form of weight $k-\frac{1}{2}$ for $\Gamma(2)$.
Clearly this function is a polynomial in $\frac{1}{v}$ with weakly holomorphic coefficients, since $D^{(s)}_j$ are almost holomorphic modular forms, $\rho^{(s)}_{M,\ell}$ is weakly holomorphic and the action of the raising operator to a weakly holomorphic function gives a polynomial in $\frac{1}{v}$ with weakly holomorphic coefficients. It remains to prove that it transforms like a (vector-valued) modular form. By definition this function can be written as 
\[ 
\widehat{h}_\ell(\tau)
+\Sigma_{\frac{1}{2},\ell}+\Sigma_{\frac{3}{2},\ell},
\]
where
\[
\Sigma_{\frac{1}{2},\ell}:=\sum_{z_s(\tau) \in S_0^{(\tau)}} \sum_{h=0}^{\ql{\frac{n_s-1}{2}}} \frac{D_{2h+1} ^{(s)}(\tau)}{(2h)!}\rlf{M}{\pi}^h 
R^{h}_{\frac{1}{2}} \rl{\widehat{\rho}^{(s)}_{M,\ell}(0;\tau)}
\]
and
\[
\Sigma_{\frac{3}{2},\ell}:=\sum_{z_s(\tau) \in S_0^{(\tau)}} \sum_{h=0}^{\ql{\frac{n_s}{2}}-1} \frac{D_{2h+2} ^{(s)}(\tau)}{(2h+1)!}\rlf{M}{\pi}^h 
R^{h}_{\frac{3}{2}} \rl{\delta_\epsilon\ql{\widehat{\rho}^{(s)}_{M,\ell}(\epsilon;\tau)}_{\epsilon=0}}.
\]
We already know that $(\widehat{h}_\ell(\tau))_{\ell \pmod{2M}}$, $(\Sigma_{\frac{1}{2},\ell})_{\ell \pmod{2M}}$, and $(\Sigma_{\frac{3}{2},\ell})_{\ell \pmod{2M}}$ transform as vector-valued modular forms of weight $k-\frac{1}{2}$ for $\Gamma(2)$. We claim that the $2M\times 2M$ matrices associated to these three functions are the same.  
From Proposition \ref{pro:1}, and more generally from the transformation properties of $\widehat{\mu}_{\ell,M}$, it follows that 
\[
\delta_\epsilon\ql{\widehat{\rho}^{(s)}_{M,\ell}(\epsilon;\gamma\tau)}_{\epsilon=0}=
(c\tau+d)^{\frac{3}{2}}\sum_{p  \pmod{2M}}a_{\ell,p}(\gamma)
\delta_\epsilon\ql{\widehat{\rho}^{(s\gamma)}_{M,p}(\epsilon;\tau)}_{\epsilon=0}
\]
and
\[
\widehat{\rho}^{(s)}_{M,\ell}(0;\gamma\tau)=
(c\tau+d)^{\frac{1}{2}}\sum_{p  \pmod{2M}}a_{\ell,p}(\gamma)
\widehat{\rho}^{(s\gamma)}_{M,p}(0;\tau)
\]
for some coefficients $a_{\ell,p}(\gamma)$.
Since they are both harmonic Maass forms, and hence modular, applying the raising operator $R^{2h}_*$ yields a form with the same transformation law (same multipliers and characters), with the weight increased by $2h$. Indeed, it is a well known fact that the raising operator commutes with the slash-operator as 
\[
R_k\rl{f|_k \gamma}=R_k\rl{f}|_{k+2}\gamma.
\]
On the other hand 
\[
D_{j} ^{(s)}(\gamma\tau)=(c\tau+d)^{k-j} D_{j} ^{(s\gamma)}(\tau).
\]
From that it follows that the entries of the transformation matrices for $(\Sigma_{\frac{1}{2},\ell})_{\ell \pmod{2M}}$ and $(\Sigma_{\frac{3}{2},\ell})_{\ell \pmod{2M}}$ are the same, namely $a_{\ell,p}(\gamma)$. On the other hand, we already know that
\[
\widehat{h}_{\ell}(\gamma\tau)=(c\tau +d)^{k-\frac{1}{2}}
\sum_{p  \pmod{2M}}b_{\ell,p}(\gamma)\widehat{h}_{p}(\tau).
\]
To finish the proof it is enough to point out that $b_{\ell,p}(\gamma)=a_{\ell,p}(\gamma)$.
Indeed, from (4.2) in \cite{Zwegers1} it is straightforward to see that
\[
\sum_{\ell  \pmod{2M}}\widehat{\rho}^{(s)}_{M,\ell}(0;\tau)\theta_{M,\ell}(z;\tau)
\]
transforms as a Jacobi form of weight $1$ and index $M$ for $\Gamma_0(2)$, exactly as $\widehat{\phi}^{F}(z;\tau)$. In particular their Fourier coefficients are the functions $\widehat{\rho}^{(s)}_{M,\ell}(0;\tau)$ and $\widehat{h}_{\ell}(\tau)$, respectively, and they transform as vector-valued modular forms with respect to the same transformation matrix. 
This prove the theorem since an almost holomorphic modular form is an almost harmonic weak Maass form. 
\end{proof}

\end{document}